%% file: mainn.tex
\documentclass[10pt,amssymb,amsfonts,psfig]{amsart}
\usepackage{amsfonts,amssymb,color,amsmath,amscd}
\usepackage{graphicx} 
\input{macros.tex}
\begin{document}

\bigskip\bigskip

\title[Unknotting number for Lorenz knots]{Unknotting number for Lorenz knots}

\author {Lilya Lyubich}
\date{\today}

 \begin{abstract} 
The unknotting number of a positive braid with $n$ strands and $k$ 
intersections is
 known to be equal to
$(k-n+1)/2$. We consider  Lorenz knots (which are  
positive 
braids) and, using a different method, find their unknotting numbers 
in terms of their positions on the Lorenz attractor.
\end{abstract}

\setcounter{tocdepth}{1} 

\maketitle
\tableofcontents

\input{introintros.tex}

\input{introductions.tex}

\input{backgrounds.tex}

\input{LtoG.tex}

\input{A-gradings.tex}

\input{intersections.tex}

\input{unknots.tex}

\input{bib.tex}
\end{document}

%% file: macros.tex
\newtheorem{thm}{Theorem}[section]

\newtheorem{lem}[thm]{Lemma}
\newtheorem{prop}[thm]{Proposition}

\theoremstyle{remark}

\newtheorem{definition}{Definition}

\newtheorem{example}{Example}[section]

\theoremstyle{definition}

\newtheorem*{LY Theorem}{Lee-Yang Theorem}

\numberwithin{equation}{section}
\numberwithin{figure}{section}

\font\nt=cmr7

\def\note#1
{\marginpar
{\nt $\leftarrow$
\par
\hfuzz=20pt \hbadness=9000 \hyphenpenalty=-100 \exhyphenpenalty=-100
\pretolerance=-1 \tolerance=9999 \doublehyphendemerits=-100000
\finalhyphendemerits=-100000 \baselineskip=6pt
#1}\hfuzz=1pt}

\newcommand{\eps}{{\varepsilon}}

\newcommand{\R}{{\Bbb R}}

\def\B0{{\mathbf{0}}}

\catcode`\@=12

\def\Empty{}
\newcommand\oplabel[1]{
  \def\OpArg{#1} \ifx \OpArg\Empty {} \else
  	\label{#1}
  \fi}
		
\newcommand{\comm}[1]{}
\newcommand{\comment}[1]{}

%% file: introintros.tex
\section{Introduction}

In this paper we calculate unknotting numbers for some knots arising in
dynamical systems.

Given a flow $\phi_t$ on a 3-manifold having a hyperbolic structure on its
chain recurrent set, the link of periodic orbits of $\phi_t$   is in bijective
correspondence with the link of periodic orbits (up to at most two exceptional
orbits) on a particular branched 2-dimensional manifold called template or the
knot-holder (\cite{BW2},Theorem 2.1).

By a template (a knot-holder) $H$ is meant a branched 2-manifold $H \subset
M^3 $, $ H \neq \emptyset $,  together with a semi-flow $\overline{\phi}_t$ 
on $H$
such that $H$ has an atlas consisting of 2 types of charts, a joining chart
and a splitting chart, as on Figure 1.  
\begin{figure*}[h]

\vspace{2mm}

 \includegraphics[height=1in,viewport=75 640 360 750,clip]{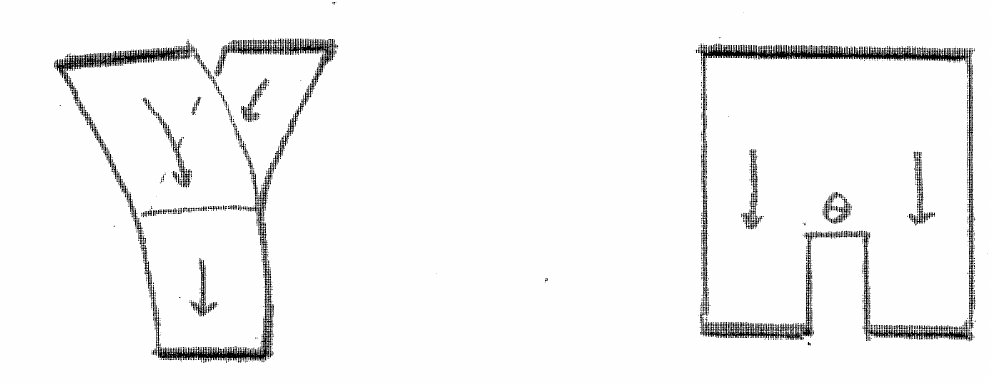}
\vspace{2mm}

Figure 1.
\label{lfig1}
\end{figure*}
Each component of the intersection of
two charts is the bottom of one and a top of the other. The flow
$\overline{\phi}_t$ is shown on the Figure 1. From the set $\theta$ the flow
goes outwards.

We chose the simplest template, the Lorenz attractor, see Figure 2,
corresponding to the flow in $S^3$ generated by the Lorenz equations 
(\cite{BW1}).
\begin{figure*} 
\includegraphics[height=5cm,viewport=0 0 515 280,clip]{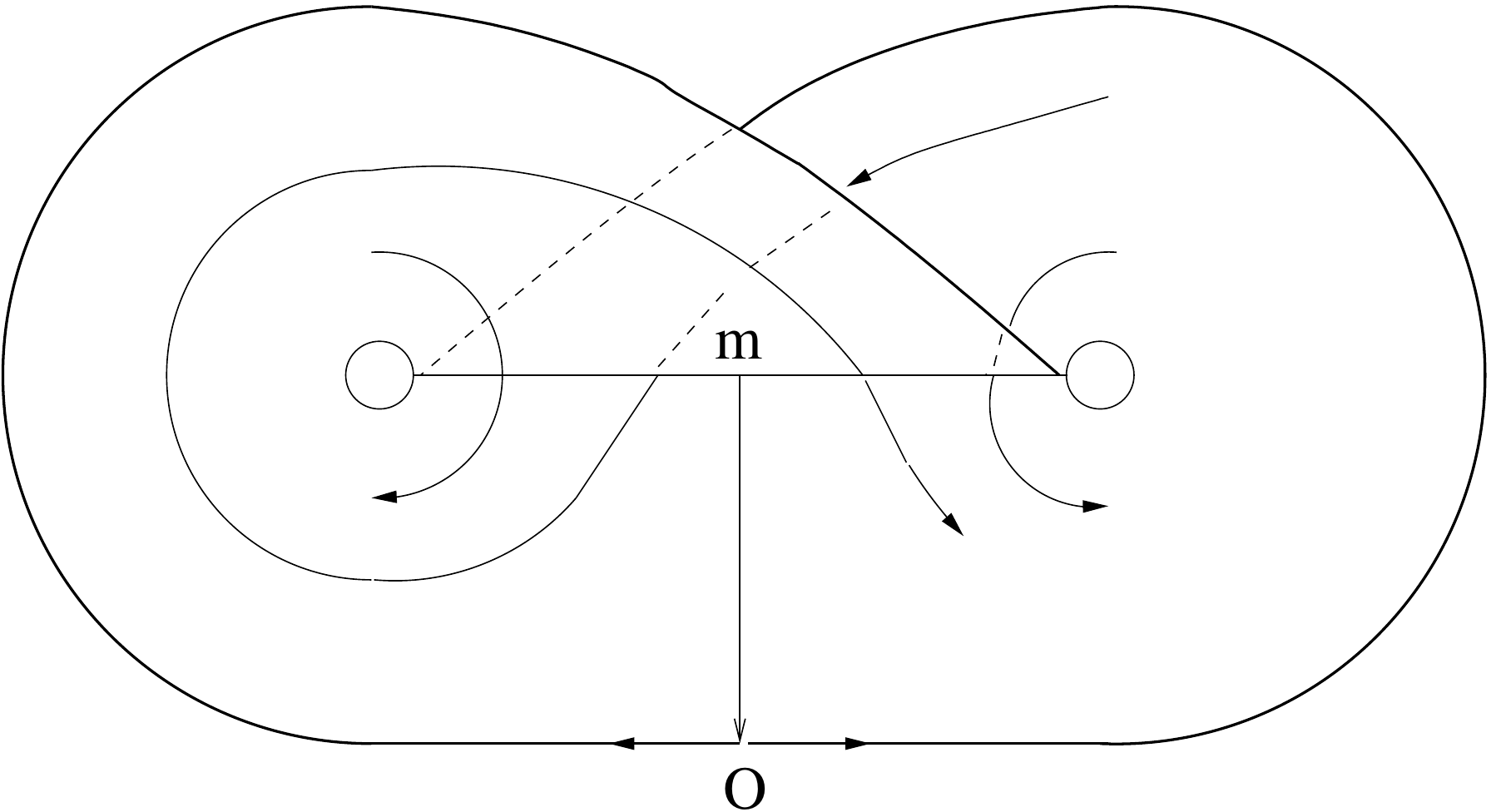}
\vspace{2mm}

Figure 2. 
\label{cfig2}
\end{figure*}
This template has an atlas consisting of two charts: a joining chart and a
splitting chart. Knots living on this template are called Lorenz knots.
All Lorenz knots are positive braids (\cite{BW1}), and it was proven by 
Charles Livingston 
in \cite{L}, 
Corollary 7, that $\tau$-invariant of a positive braid is equal to
$ (k - n + 1) / 2 $, where $k$ is the number of  intersections, and $n$
is the number of  strands of the braid. This gives a lower bound for the 
unknotting number, and since it is also an upper bound (\cite{BoW}, \S 4),
the unknotting number of a positive braid is equal to
$(k-n+1)/2$.
 
In this paper we calculate, using a different method, 
the unknotting numbers of Lorenz knots in terms of their positions on Lorenz
attractor.
Here is an outline of the paper. In section 2 we use symbolic dynamics
and introduce an additional geometric description of Lorenz knots to formulate
the main result of this paper, Theorem 2.3. In section 3 we give necessary
background in grid homology and $\tau$-invariant. Sections 4-7 contain the
proof of Theorem 2.3.                                                          
          
Acknowledgment. I am very grateful to Peter Ozsv\'ath for introducing me to 
 grid
homology and $\tau$-invariant, to Misha Lyubich for suggesting me to calculate
$\tau$-invariant for knots arising in dynamical systems, and to Scott
Sutherland for his help with figures.

%% file: introductions.tex
\section{Lorenz knots}

Let $ H $ be the Lorenz attractor, see Figure 2.
Denote by $I$ the branch set of $H$. The positive orbit through $ m $, $\bar{\phi } _t(m),\: t>0 $ approaches $O$ as $t \rightarrow \infty $.
For $ x \in I - \{ m \} $ there is a first return (Poincar\'e) image, $ f(x) \in I $. The orbit from $ x $ to $ f(x) $ goes
around the left
hole and in front for $ x<m $, around the right hole and in back for $ x>m $. Since the orbits don't intersect, $ f(x) $ is
monotonically increasing on both sub-intervals : $I_1=\{x \in I,\; x<m \} $ and $ I_2=\{x \in I, \; x>m \} $.

 Consider symbolic dynamics on $I$  generated by the first return map $f$.
To each point $z$ we assign a finite or infinite sequence $k_0(z), k_1(z), \ldots $ where
 $$
k_0(z)=\begin{cases} x  \mbox{ if }  z \mbox{ is to the left of } m \\
          0 \mbox{ if } z=m \\
          y \mbox{ if } z \mbox{ is to the right of } m; \end{cases}
$$
and $k_i(z) $ is defined  iff $f^i(z) $ is defined by  $k_i(z)=k_0(f^i(z)).$
The sequences $k$ are lexicographically ordered by setting $ x<0<y$.
\begin{prop}
(\cite{BW1}, Proposition 2.4.1)

\noindent
The map $ z \rightarrow k(z) $  is a 1-to-1 order preserving correspondence between the points of the branch set and
 the set of all sequences $ k_0,k_1, \ldots $ such that
$$
\begin{array}{l}
(a)\;  \mbox{ each }\; k_i=x,y \mbox{ or } 0 \\
(b) \; \mbox{ the sequence terminates with } k_i \mbox{ iff } k_i=0.
\end{array}
$$
\end{prop}
\begin{thm}(\cite{BW1} Corollary 2.4)

\noindent
The periodic orbits of $ \bar{\phi } _t $ correspond 1-to-1 with the cyclic permutation classes of finite aperiodic  words
in $ x $ and $ y $.
\end{thm}
\begin{example}  On Figure 3 we have a  diagram of the knot $ K=x^3y^3xy^2.$
\end{example}

\begin{figure*} 
\includegraphics[height=5cm,viewport=25 490 590 755,clip]{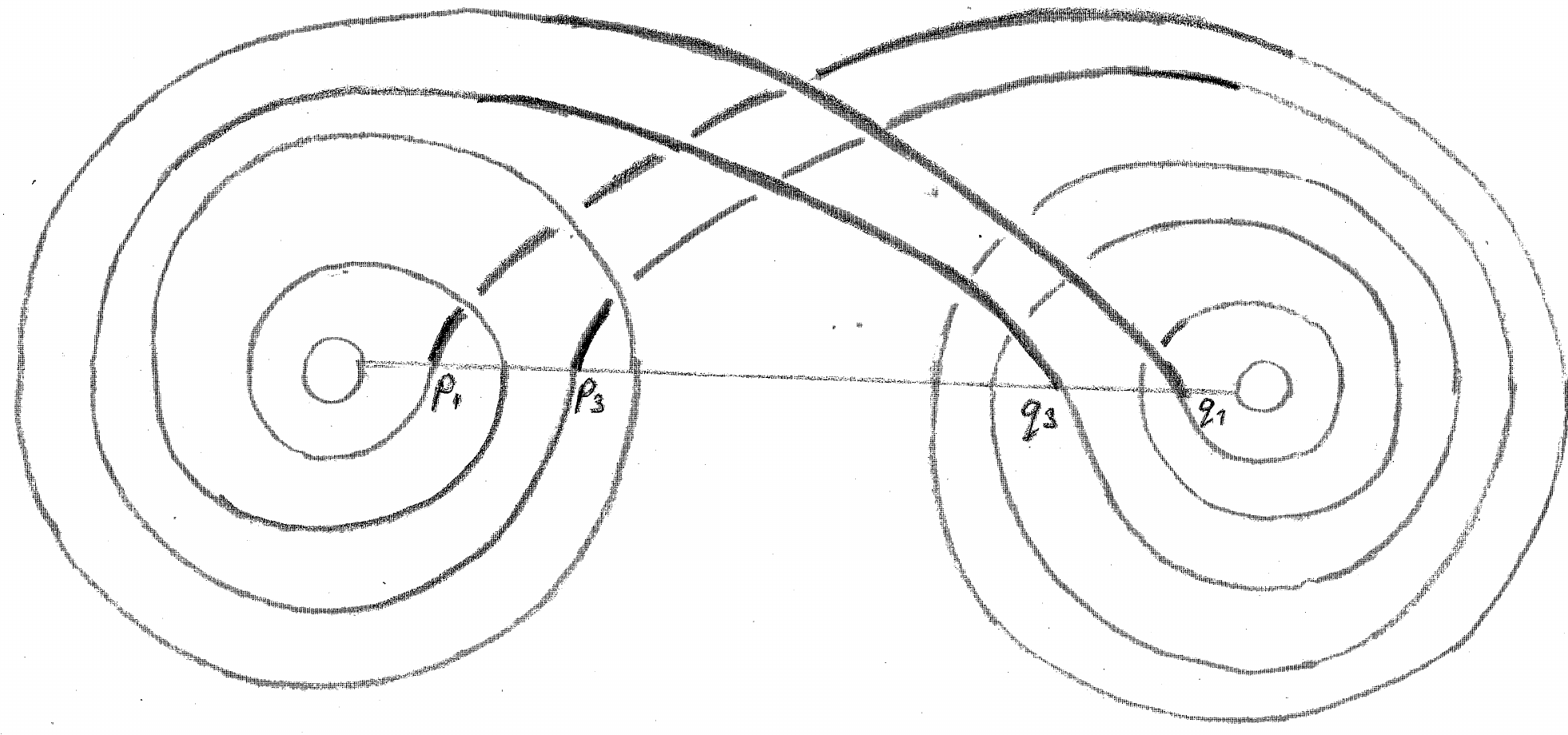}
\vspace{2mm}

Figure 3. 
\label{lfig3}
\end{figure*}

Given a Lorenz knot K corresponding to the word $ x^{\alpha _1} y^{\beta _1} x^{\alpha _2} \ldots x^{\alpha _t} y^{\beta _t} $, let $a=\sum _1^t \alpha _i$,
$ b=\sum _1^t \beta _i $, and $ t $ be the trip number: the number of strands going from the  left to the right ( equal to the number of
strands going from the right to the left). To formulate our result we need to introduce some notations. We want to describe the
positions of the strands going from the right to the left, $ \mu _i $, and from the
left  to the right , $ \nu _i $.
To the knot $ K $ there corresponds a finite set of its intersection points with the branch set $ I $. Denote them by
$ p_1 < p_2 \ldots < p_a <  q_b < \ldots < q_1; \;  p_i \in I_1, \; q_i \in I_2. $ These points are permuted by the first return map .
We have: either $ p_i= f( p_j) $ or $ p_i =f(q_l) $ . We are interested in the case when $ p_i =f(q_l) $
for some $ l $. Let $  \{p_1, \ldots p_a \}  \cap f(\{q_b, \ldots , q_1\})=\{p_{\mu _1}< p_{\mu _2}< \ldots < p_{\mu _t}\}$.
 Similarly, let $  \{ q_b, \ldots , q_1 \} \cap f( \{ p_1, \ldots ,p_a\})= \{q_{\nu _t} < \ldots, < q_{\nu _1} \} $.
 This defines $\mu _i $ and $\nu _i $ uniquely.
Note that $ \mu _1 =1 $ , since  if $ p_1 = f(p_j)$ for some $ j $, by monotonicity of $ f $ , $ f(p_1) $ must be to the left
of $ p_1 $ which contradicts our numeration.  So $ p_1 = f(q_l) $ for some $ l $. For the same reason $ \nu _1 =1 $.
\begin{example} For the knot $K$ on Figure 3 we have $\mu _1=1, \:\mu _2=3, \:\nu _1=1, \:
\nu _2=3$.
\end{example}

\noindent
We prove the following theorem:
\begin{thm} The unknotting number of a Lorenz knot $ K $ is given by the formula
\begin{equation}\label{e1}
 u(K)= \frac{1}{2}((a+b)(t-1)-\sum _1^t \mu _i -\sum _1^t \nu _i +(t+1)).
\end{equation}
\end{thm}
The plan of the proof: we turn the diagram of a Lorenz knot into a grid diagram.
We find a lower bound on unknotting number using the $ \tau$-invariant of the knot.
We show that this number of crossing changes is enough to unknot the knot.

%% file: backgrounds.tex
\section{Grid diagram and $ \tau $-invariant of a knot, background}
There is a
theory of grid homology for knots and links, introduced by Peter S. Ozsv\'ath and
  Zoltan Szab\'o. In particular they define the $\tau$-invariant of a knot  from its grid homology and prove that
it gives a lower bound for the unknotting number:
\begin{equation}\label{e2}
 |\tau (K)|\leq u(K).
 \end{equation}
Since we will not use  the definition of $\tau $-invariant we will not present it here .  Rather we will use a lower
bound for
$\tau $-invariant obtained from a grid diagram of a knot.

Grid diagram is $ n \times n $ grid on the plane with $ n $ small squares marked by $ X $ and another $ n $ small
squares marked by $ O $, so that there is exactly one $ X $ and one $ O $ in each column and in each row. In each column
we connect $ X $ to $ O $ by an oriented segment, and in each row we  connect $ O $ to $ X $ by an oriented segment.
We arrange that vertical segments go over horizontal ones. Clearly, such a diagram defines an oriented link, and every
oriented link in $ S^3 $ can be presented by a grid diagram. (See Figure 4a.)

\begin{figure*} 
\includegraphics[height=5cm,viewport= 43 475 607 750,clip]{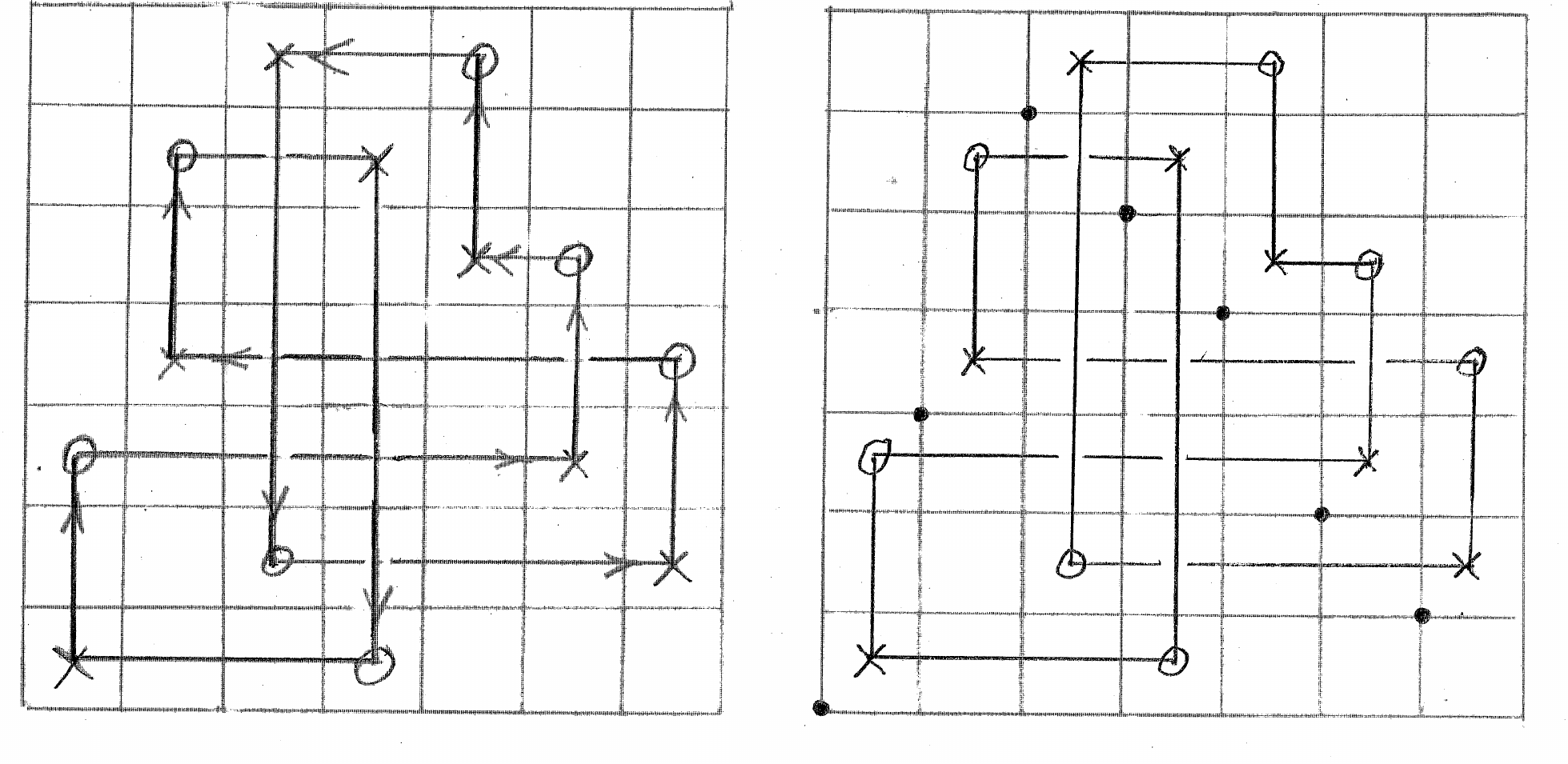}
\vspace{2mm}

Figures 4 a,b . 
\label{lfig4}
\end{figure*}

We transfer our planar grid diagrams to the torus $ T $ obtained by identifying the top-most segment with the bottom-most
 one and the left-most segment with the right-most one.
In the torus horizontal and vertical arcs of the grid become horizontal and vertical circles.

A grid state is an $ n-$tuple
of intersection points $\bold{x}=\{x_1,\ldots, x_n \}$ of vertical and horizontal circles satisfying the requirement
that each vertical circle
contains exactly one of the elements of $\bold{x} $ and each horizontal circle contains exactly one element of $\bold{x} $.
Let $ S(G )$  denote the set of grid states of the grid diagram $ G $.
There is a bi-grading on $ S(G)$: Maslov and Alexander gradings.
\begin{definition} (\cite{OSS}, Definition 4.3.1) Let $ P, Q $ be two collections of finitely many points in the plane $\R ^2$. Let $I(P,Q) $
be the number of pairs $ ( p_1,p_2) \in P $ and $ (q_1,q_2) \in Q $ with  $ p_1<q_1 ,\; p_2<q_2 $.
We symmetrize this, defining $$ J(P,Q)=\frac{I(P,Q)+I(Q,P)}{2}. $$ Then we extend $ J$ bi-linearly over formal sums of subsets
of the plane.
\end{definition}
We will need only the Alexander grading. It can be calculated by the formula:
\begin{equation}\label{e2.2}
 A(\bold{x})= -\sum _{x\in \bold{x}} w_K(x) + \tfrac{1}{2} (J(O,O)-J(X,X))-\tfrac{(n-1)}{2},
 \end{equation}
where $w_K(x) $ is the winding number of the knot around the point $ x $ and $ n $ is the grid number. We will be interested
in a particular grid state $ \bold{x^-} $ which occupies the lower left corner of each square marked with $ X $. (See Figure 4b.) 
For this state it  is proven (\cite{OSS} Proposition 6.3.15) that
\begin{equation}\label{e2.4}
 A(\bold{x}^-) \leq -\tau (K).
 \end{equation}

%% file: LtoG.tex
\section{From Lorenz diagram to grid diagram}
It is easy to turn the diagram of a knot on the Lorenz template into a grid diagram, see Figure 5.

\begin{figure*}
\includegraphics[height=5cm,viewport=65 390 570 745,clip]{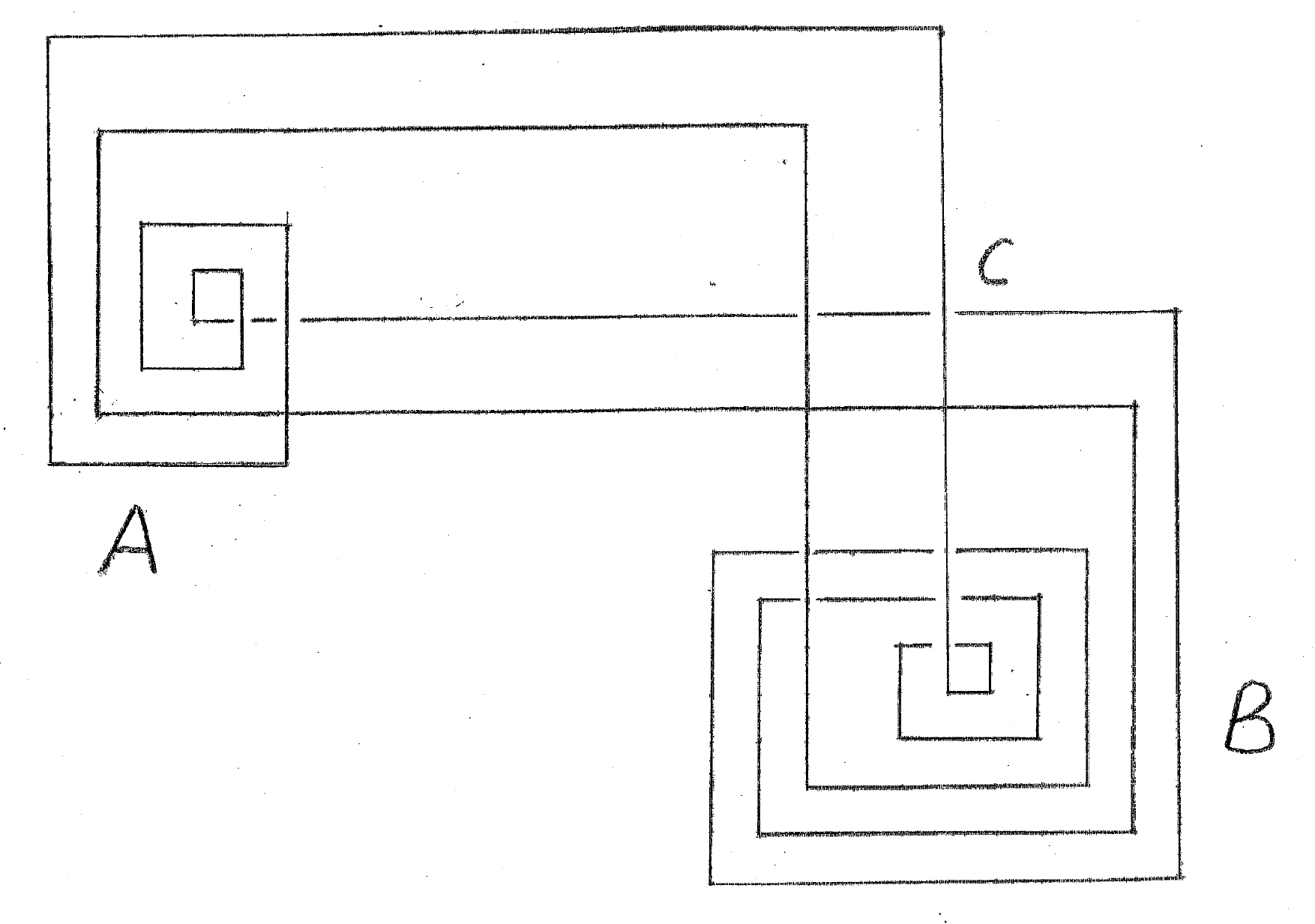}
\vspace{2mm}

Figure 5.
\label{lfig5}
\end{figure*}

The strands going from left to 
right become long vertical strands, and the strands going from  right 
to  left become long horizontal strands. We divide our
diagram  onto upper left part $ A $, lower right part $ B $ and the 
central square $ C $, where
 the vertical and horizontal strands intersect. Note that all 
intersections on the diagram
 happen on long vertical and long horizontal
strands since short vertical and horizontal segments don't intersect 
each other (except for the end points).
Here we used the following convention:
on the grid diagram all vertical segments going from $ X $ to $ O $ fit 
into two groups: long vertical strands going from $ A $ to $ B $ and 
short vertical segments belonging either to $ A $ or to $ B $.
Similarly, all horizontal segments going from $ O $ to $ X $ are either 
long horizontal strands going from $ B $ to $ A $, or short horizontal 
segments belonging to either $ B $ or $ A $.

Note that the grid number of the constructed diagram is $ n=2a+2b-t. $ 

%% file: A-gradings.tex
\section{Lower bound for unknotting number}
 Let $ \bold{x}^- $ be the grid state occupying lower left corners
 of squares marked with $ X $, Figure 6.
\begin{figure*} 
\includegraphics[height=6cm,viewport=60 390 535 745,clip]{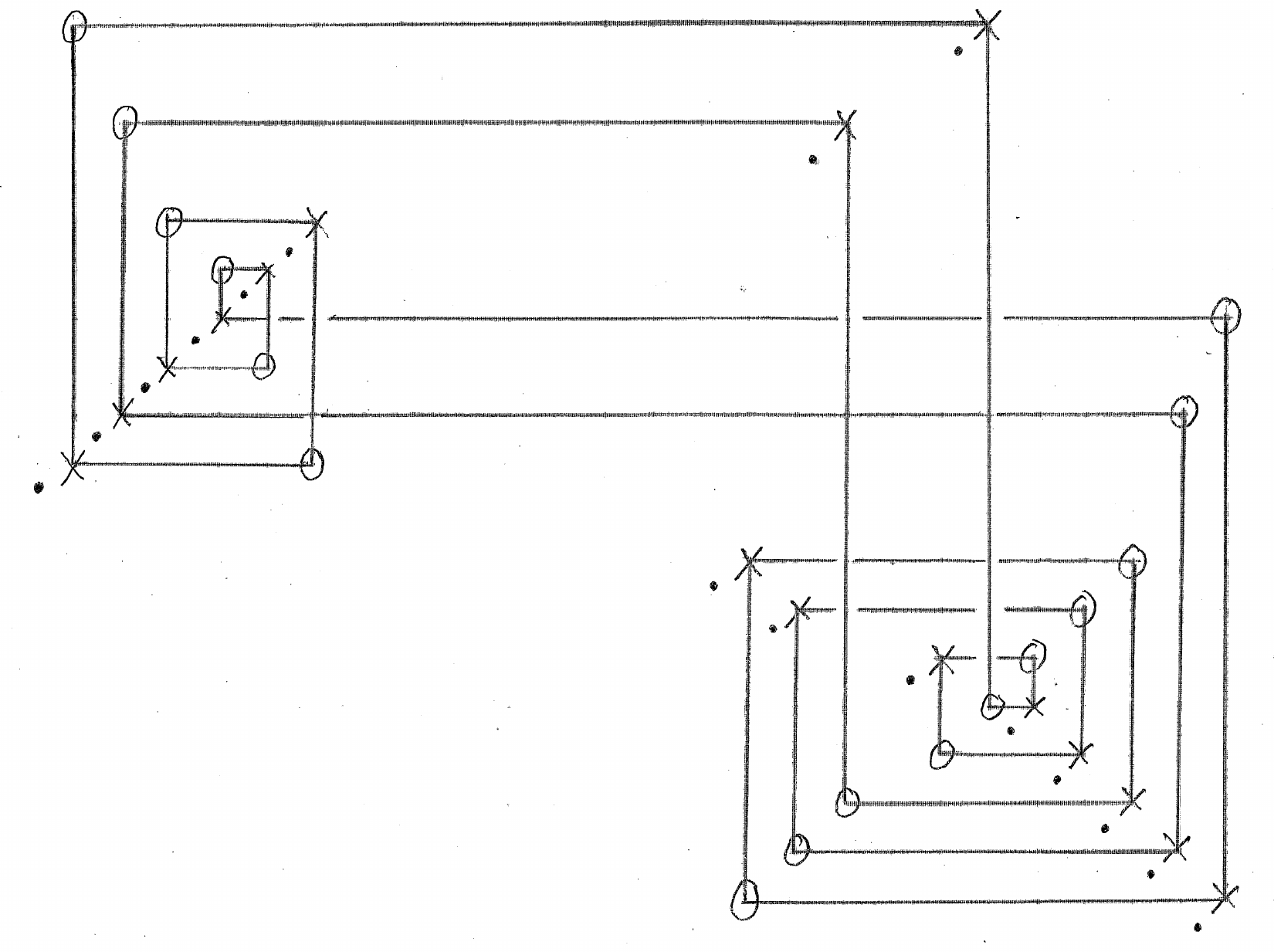}
\vspace{2mm}

Figure 6. 
\label{lfig6}
\end{figure*}

We calculate  its Alexander grading using the formula
\begin{equation}\label{e4.1}
 A(\bold{x}^-)= -\sum _{x\in \bold{x}} w_K(x) + \tfrac{1}{2} (J(O,O)-J(X,X))-\tfrac{(n-1)}{2},
\end{equation}

 \begin{thm}
 \begin{equation}\label{e4.2}
 A(\bold{x}^-)= \frac{1}{2}((a+b)(t-1)-\sum _1^t \mu _i -\sum _1^t \nu _i +(t+1)).
 \end{equation}
\end{thm}
\begin{proof}
First we calculate $ A'(\bold{x}^-)=-\sum _{x \in \bold{x}^-} w_K(x)=-\sum _A w_K(x)- \sum_B w_K(x)$, where $\sum _A $ is the sum over upper $ 2a $ elements of $ \bold{x}^- $ and $\sum _B $ is the sum over lower $2b -t $ elements of $ \bold{x}^- $, see Figure 6.
Since for  $ x \in A $ , $ -w_K(x) $ is equal to the total number of vertical segments to the left of it,
$$ -\sum _A w_K(x) = 0+1+2+\ldots +(a-1)+a+(a-1)+ \ldots +1=a^2.$$
To calculate $-\sum _B w_K(x) $ note that for $b $ rightmost points $ x\in \bold{x}^- $, $ -w_k(x) $ is equal to negative
 number of vertical segments to the right of $ x $, so these $ x $ contribute $-(0+1+\ldots +b-1) =-\frac{b(b-1)}{2} $.
 Only $b-t $ points remain to the left of the first vertical strand. For them it is convenient to calculate $-w_K $ as negative number of vertical segments and vertical strands to the left of them.
 They contribute $0+1+\ldots+(b-t-1) $ plus vertical strand number $k $ adds $1$ to  $\nu _k -1-(k-1) $ points ( There are $\nu _k -1 $ vertical segments to the right of strand number $\nu _k $ till the first strand and $k-1 $ of them are long strands. The number of points $x\in \bold{x}^- $ in the fragment of the diagram that we consider now is the number of short segments).

$$
  -\sum _B w_K (x)=-\frac{b(b-1)}{2}-\frac{(b-t-1)(b-t)}{2} -\sum_1^t ((\nu -1)-(k-1))
  $$
  $$
 =-b^2+b+bt-\sum_1^t \nu_k ,
 $$

 and $$A'(\bold{x}^-)=-\sum_{x\in\bold{x}^-} w_K(x)=a^2-b^2+b+bt-\sum_1^t \nu _k.$$

 Now $ J(O,O) $ and $J(X,X) $ are calculated straightforwardly. There are $ 2b\; $ $O'$ s in part $B \cup C $
 that contribute $ b(2b-1) $. In part $ A \cup C $ each $ O $ on the $k-$th long horizontal strand that has position $ \mu _k $ contributes  the amount equal to the amount of short horizontal segments below this strand. There are $ a- \mu _k $ horizontal segments below $\mu _k $ and $ t-k $ of them are long horizontal strands. So there are $a-\mu _k-(t-k) $ short horizontal segments ($ O'$s ) below $ k-$th strand. We get
 $$
 J(O,O)=\sum _{k=1}^t(a-t-(\mu _k-k))+b(2b-1)=at-\sum _1^t \mu _k-t^2+\tfrac{t(t+1)}{2} +b(2b-1).
 $$
 Similar argument gives
 $$
 J(X,X)=\sum _{k=1}^t(b-t-(\nu _k -k)) +a(2a-1)
 $$
 $$
 =bt-\sum _1^k \nu _k -t^2+ \tfrac{t(t+1)}{2}+2a^2-a.
 $$
 Plugging all these results into formula \eqref{e4.1}, we get the desired formula \eqref{e4.2}.
 \end{proof}

%% file: intersections.tex
\section{Intersection points of a grid diagram}
\begin{lem}
The number of crossings on the grid diagram for $ K $ is
\begin{equation}\label{e5.1}
c=(a+b)t -\sum \mu _k -\sum \nu _k +t
\end{equation}
\end{lem}
\begin{proof}
Note that all intersections happen on long vertical and long horizontal  strands. Take a vertical strand and calculate the number  of intersections on part $ B $ of it. That is the number of short horizontal segments that go down after intersecting the strand. It is equal to the total number of vertical segments to the left of our strand minus the number of the long vertical strands to the left of it. If the $ k $-th strand occupies position $ \nu _k $, then the number of intersections on it is $$ c_k=(b-\nu _k) -(t-k) $$ Taking the sum over $ k $ we get
$$
c_B=bt-\sum_1^t\nu _k -t^2 + \frac{t(t+1)}{2},
$$
where $c_B $ is the number of intersections in part $B$.
Similarly, for horizontal strands
$$
c_A=at-\sum_1^t\mu _k -t^2 + \frac{t(t+1)}{2},
$$
where $c_A$ is the number of intersections in part $A$.
Then we add $ t^2 $ intersections of long vertical strands with long horizontal strands to get
total number of intersections
$$
c=(a+b)t-\sum_1^t \mu _k -\sum_1^t\nu _k +t
$$
\end{proof} 

%% file: unknots.tex
\section{Unknotting the knot}
\begin{thm}
Unknotting number of a Lorenz knot $K ,\; u(K) $, satisfies the inequality
 \begin{equation}\label{eq6.1}
 u(K) \leq  \frac{1}{2}((a+b)(t-1)-\sum _1^t \mu _i -\sum _1^t \nu _i +(t+1)).
 \end{equation}
\end{thm}
\begin{proof}
It is a classical observation that if we trace a diagram of  a knot going first time over and second time under at each intersection, then it is a diagram of the unknot. Call such a diagram $ U $-diagram. We start to trace our diagram from $ X $ of the first long horizontal strand and count the number of crossing changes needed to turn it into $ U $-diagram. We call the crossing that we don't need to change
"right crossings", and these we have to change  "wrong crossings". Let us cut long horizontal strands at their  left ends marked $ X $. We obtain $ t $ strings. Numerate them
from $1 $ to $ t \;$  in the order we trace the diagram  . For example, the first long horizontal strand belongs to the $ t $-th string.
In the word for this knot the $i-$th string corresponds to the part $x^{\alpha _i}y^{\beta _i}.$
Each string goes around in $ A $-part without self intersections, then goes down  the vertical strand, goes around in $ B $-part, probably intersecting itself on the long vertical strand then returns to $ A $-part via the long horizontal strand  where it may intersect itself. It shows that self intersections of each string are right crossings, because we trace a vertical segment of a string before tracing the corresponding horizontal segment. The number of self intersections on the long vertical strand of the $ i $-th string is $ \beta _i $, so the total number  of self intersections on $ B \cup C $ is
$\sum_1^t \beta _i =b $.

  Next we analyze intersections between $i $-th and $j $-th strings inside $B \cup C $.
  Note that a crossing $ (i,j)$ with vertical segment belonging to the $i $-th string and the horizontal segment belonging  to the $ j$-th string is right if $ i< j $ and is wrong if $ i>j $.

  There are two possible positions of these strings in $ C $, corresponding to the two possible scenario  drawn on Figure 7a,b.
\begin{figure*} 
\includegraphics[height=6cm,viewport=0 0 540 400,clip]{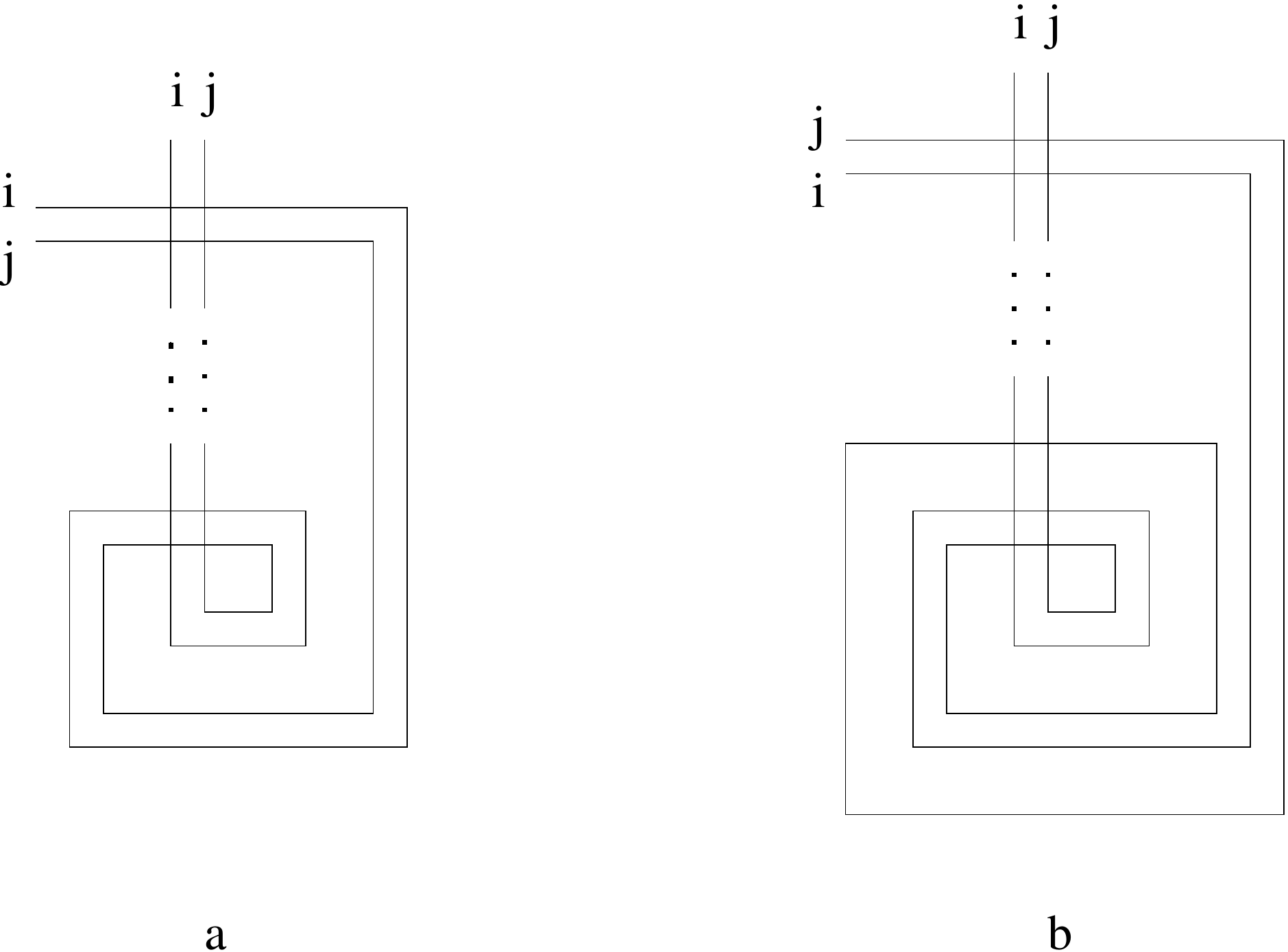}
\vspace{2mm}

Figure 7  
\label{fig7b}
\end{figure*}



  We introduce $\eps _{ij}=1, -1 \; \text{or } 0 $ according to the rule:
 \begin{equation*}
 \eps _{ij}=\begin{cases}
  $\; 1 $,\text{ if in the upper left corner of the rectangle formed by long vertical}\\
   \;\;\;\; \text{    and long horizontal strands of $ i $-th and $j $-th strings, the vertical} \\
   \;\;\;\; \text{    string is greater than the horizontal;}\\
  -1,  \text{ if the vertical string on that corner is less than the horizontal;}\\
  $\; 0$, \text{ if they are equal.}
   \end{cases}
  \end{equation*}

  Let $c_{ij} $ be the number of intersections between the strings $i, j$ in part $B\cup C $ (self intersections are not included). Then the number of wrong intersections  among them is $\frac{1}{2} (c_{ij}+ \eps _{ij}) $. By $\sum _{(i,j)}$ we mean the sum over unordered pairs $(i,j)\; 1 \leq i,j \leq t,\;i \neq j.$  Denote $ N_B =\sum_{(i,j)} \eps  _{ij}$.
  Summing up over all pairs of strings we get  the number of wrong crossings  in part $B\cup C$ to be $ u_{B \cup C} =\frac{1}{2} (c_{B \cup C } -b +N_B)=\frac{1}{2}(b(t-1) -\sum_1^t \nu_k + \frac{t(t+1)}{2}  +N_B )$

Calculation of the number of wrong intersections in part $ A \cup C $, $u_{A \cup C}  $, is similar. We cut  each long horizontal strand at its right end $ O $, so that the first long horizontal strand that belonged to the $t $-th string now belongs to the first string , and $ k $-th long horizontal strand if belonged to $ (i-1) $-th string now belongs to $i $-th string. Denote the central square with the new numeration by $C_A$ , and with the old numeration by $C_B$. Now in each string long horizontal strand goes before long vertical strand. We can calculate the number of intersections $(i,j) $ where horizontal segment $ j $ is greater then vertical segment $i $, using the same method we used in $ B $-part
(for this calculation   horizontal strands play the role of vertical strands in $B $).
Let




\begin{equation*}
\delta _{ij}=
\begin{cases}
\;  1, \text{ if in the lower right corner  of the rectangle formed}\\
\;\;\;\;\text{  by  $i $-th and $j $-th strings in $C_A$, the horizontal}\\
\;\;\;\;\text{ string is greater than the  vertical string,}\\
-1,  \text { if the horizontal string  is smaller then the vertical one,}\\
\; 0, \text{ if they are equal.}
\end{cases}
\end{equation*}

\noindent
Let $N_A=\sum_{(i,j)}\delta _{ij}$.
The number of intersections $(i,j)$ in $A \cup C $ with $j>i $ is $$
 u_{j>i}=\frac{1}{2}[c_{A \cup C}-a+N_A].
 $$
 These are right intersections. Also all $a $ intersections  for $j=1$ are right. All the rest intersections: with $i>j $ and $i=j \neq 1 $ are wrong, and there are $u_{A \cup C}$ of them:
 $$u_{A \cup C}=c_{A \cup C }-\frac{1}{2}(c_{A \cup C}-a+N_A)-a=\frac{1}{2}(c_{A \cup C}-a-N_A)$$
 $$=\frac{1}{2}(a(t-1) -\sum_1^t \mu_k + \tfrac{t(t+1)}{2}  -N_A ) $$
 The  number of wrong intersections in $C $ is clearly $\frac{t(t-1)}{2}$ so the total number of wrong intersections on the diagram is $ U= u_{A\cup C}+u_{B\cup C}- \frac{t(t-1)}{2}$, so
 $$
 u(k)\leq U=\frac{1}{2}((a+b)(t-1)-\sum_1^t\mu _k -\sum_1^t\nu _k +2t+N_B-N_A).
 $$
To prove Theorem(6.1) it remains to prove
 that
   $ N_B-N_A=-(t-1) $.

  The proof of this statement uses only some combinatorics on the square and is not related to a knot.

 Suppose we have two squares $A $-square and $ B $-square with $ t $
 vertical strands, numerated from $1$ to $t $, and $ t $ horizontal strands numerated from $1$ to $t $, so that in $B $-square
 the uppermost horizontal strand has number $ t $. In $A$-square  the numeration of vertical strands coincides with that of $ B $-square , the uppermost horizontal strand changes number from $t $ to $1$ , and all other horizontal strands if had number $ i $, get number  $ i+1 $. (A- and B-squares inherit numeration of the strands from the cuts of the long horizontal strands  at O's and at X's respectively).
 \begin{lem}
 When we interchange vertical strands, $N_B-N_A $ does not change.
 \end{lem}
 \begin{proof}
 In the picture we see the fragments of $ B $ and $ A $ squares  with rectangles formed by $i $-th and 
 $ j $-th strands.
 These two configurations describe all possible mutual positions of $i $-th and $j $-th strands up to interchanging $i $-th and $j $-th vertical strands, which is an involution. So it is enough to prove the statement for these two configurations.
 When we interchange vertical strands on Figure 8a,
$\eps_{ij} $ and $\delta_{ij}$ change from $0$  to $1$ if $i>j$
and from $0$ to $-1$ if $i<j$.  So $\eps _{ij}-\delta_{ij}$ does not change. 
If the position of $i $-th and $j $-th strands is as on Figure 8b then
  $\eps _{ij}$ changes from $0$ to $-1$, $\delta _{ij} $ changes from $1$ to $0$ if
$i>j$, and $\eps _{ij}$ changes from $0$ to $1$, $\delta _{ij}$ changes from -1 to 0, 
if $i<j$ and 
  $\eps _{ij}-\delta _{ij} $ again does not change.
So $N_B -N_A $ does not change when we change the order of vertical strands.
 \end{proof}

\begin{figure*} 
\includegraphics[height=3cm,viewport=0 0 710 188,clip]{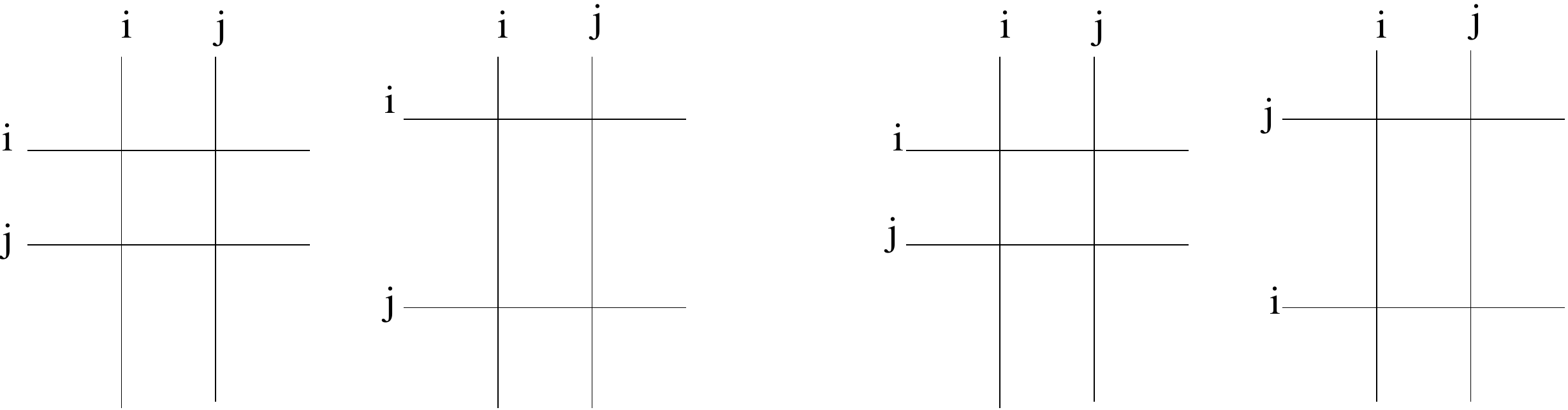}
\hspace{3mm}\text{B-square} \hspace{12mm} \text{A-square} \hspace{20mm} 
\text{B-square} \hspace{12mm} \text{A-square}\\
\vspace{7mm}
\hspace{5mm} \text{Figure 8a.} \hspace{45mm} \text{Figure 8b.}
\vspace{2mm}

\label{afig8}
\end{figure*}

 Now we prove that
$ N_B-N_A=-(t-1) $.
 Change the order of vertical strands so that they go from $1 $ to $t $. Then there is one-to-one correspondence  $(i,j) \leftrightarrow $$(i+1,j+1) $ between the intersection points of $B $-square $(i,j), \;i,j \neq t $ and intersection points of $A $-square with $i,j \neq 1$. We have $\eps _{ij} -\delta _{i+1, j+1}=0$. So $$
 \sum_{i,j \neq t} \eps _{ij} - \sum _{i,j  \neq 1} \delta _{ij}=0.
 $$ (Summation is over non-ordered pairs $(i,j)$.)
 Now $\eps _{t,j}=-1 $ for $j=1,2, \ldots ,
  t-1$ and $\delta _{1,j}=0 $ for $j=2, \ldots ,t $.
  So $ N_B - N_A =-(t-1).$
  This completes the proof of  Theorem 6.1.
  \end{proof}

We have proven $$ u(K) \leq \frac{1}{2}((a+b)(t-1)-\sum_1^t\mu _i -\sum_1^t \nu _i +(t+1))=A(x^-).$$ Combining this with inequalities $$A(x^-)\leq -\tau (K) \;, \;\;
 |\tau (K)|\leq u(K) \;,$$
we obtain  the unknotting number for Lorenz knots:
$$u(K)=\frac{1}{2}((a+b)(t-1)-\sum_1^t\mu _i -\sum_1^t \nu _i +(t+1)).$$